\documentclass[12pt]{amsart}

\newcommand{\w}{\omega}

\newcommand{\U}{\mathcal U}

\newcommand{\IZ}{\mathbb Z}
\newcommand{\IN}{\mathbb N}
\newcommand{\e}{\varepsilon}
\newcommand{\Aut}{\mathrm{Aut}}
\newcommand{\Zeta}{\mathfrak Z}

\newtheorem{theorem}{Theorem}

\newtheorem{corollary}{Corollary}
\newtheorem{example}{Example}
\newtheorem{problem}{Problem}
\theoremstyle{definition}
\newtheorem{remark}{Remark}

\title{Zariski topologies on groups}
\author{Taras Banakh, Igor Protasov}
\address{Ivan Franko National University of Lviv, Ukraine and\newline Uniwersytet Humanistyczno-Pryrodniczy Jana Kochanowskiego w Kielcach, Poland}
\address{Taras Shevchenko National University of Kyiv, Ukraine}
\email{T.O.Banakh@gmail.com}
\email{I.V.Protasov@gmail.com}

\begin{document}
\begin{abstract}
The $n$-th Zariski topology $\Zeta_{G^n[x]}$ on a group $G$ is generated by the sub-base consiting of the sets $\{x\in G:a_0x^{\e_1}a_1x^{\e_2}\cdots x^{\e_n}a_n\ne1\}$ where $a_0,\dots,a_n\in G$ and $\e_1,\dots,\e_n\in\{-1,1\}$.
We prove that for each group $G$ the 2-nd Zariski topology $\Zeta_{G^2[x]}$ is not discrete and present an example of a group $G$ of cardinality continuum whose 2-nd Zariski topology has countable pseudocharacter. On the other hand, the non-topologizable group $G$ constructed by Ol'shanskii has discrete 665-th Zariski topology $\Zeta_{G^{665}[x]}$.
\end{abstract}

\maketitle

In this paper we study topological properties of groups endowed with the Zariski topologies $\Zeta_{G^n[x]}$. These topologies are defined on each group $G$ as follows. Let $G[x]=G*\langle x\rangle$ be the  free product  of the group $G$ and the cyclic group $\langle x\rangle$ generated by an element $x\notin G$.  The group $G[x]$ can be written as the countable union
$$G[x]=\bigcup_{n\in\w}G^n[x].$$Here for a subset $A\subset G$ we put  $$A^0[x]=A\mbox{ and }A^{n+1}[x]=A^n[x]\cup\{w\,x\, a,w\,x^{-1}a\colon w\in A^n[x],\;a\in A\}\mbox{ for }n\in\w.$$
The elements of the subset $A^n[x]$ are called {\em monomials of degree $\le n$ with coefficients in the set $A$}. 


Each monomial $w\in G[x]$ can be thought as a map $w(\cdot):G\to G,\;w:g\mapsto w(g)$, where $w(g)$ is the image of $w$ under the group homomorphism $G[x]\to G$ that is identical on $G$ and maps the element $x\in G[x]$ onto $g\in G$. The subset $$\mathfrak Z_w=\{x\in G\colon w(x)\ne 1\}$$ is called the {\em co-zero} set of the monomial $w$. Here $1$ is the neutral element of $G$.

Any family of monomials $W\subset G[x]$ induces a topology $\Zeta_W$ on $G$, generated by the sub-base $\big\{{\Zeta_w}\colon w\in W\big\}$ consisting of co-zero sets of the monomials from $W$. The topology $\Zeta_{G[x]}$ is called the {\em Zariski topology} on $G$ by analogy with the Zariski topology well-known in algebraic geometry. This topology was studied in \cite{Mar}, \cite{DS}. For $n\in\w$ the topology $\Zeta_{G^n[x]}$ will be referred to as the {\em $n$-th Zariski topology} on $G$. It is clear that the Zariski topologies form a chain
$$\Zeta_{G^0[x]}\subset \Zeta_{G^1[x]}\subset \Zeta_{G^2[x]}\subset\dots\subset \Zeta_{G[x]}.$$
The 0-th Zariski topology $\Zeta_{G^0[x]}$ is antidiscrete while the 1-st Zariski topology $\Zeta_{G^1[x]}$ coincides with the cofinite topology on $G$ and hence satisfies the separation axiom $T_1$. The same is true for Zariski topologies $\Zeta_{G^n[x]}$ with $n\ge 1$.

 It is easy to see that for every $n\in\w$ the group $G$ endowed with the $n$th Zariski topology $\Zeta_{G^n[x]}$ is a quasi-topological group. The latter means that the group operation is separately continuous and the inversion is continuous with respect to the topology $\Zeta_{G^n[x]}$ on $G$.

As we already know, for an infinite group $G$, the topology $\Zeta_{G^1[x]}$, being cofinite, is non-discrete. The same is true for the topology $\Zeta_{G^2[x]}$. 

\begin{theorem}\label{t1} For every infinite group $G$ the 2-nd Zariski topology $\Zeta_{G^2[x]}$ is not discrete.
\end{theorem}

\begin{proof} Assuming the converse, we would find a finite subset $W\subset G^2[x]$ such that $\{1\}=\bigcap_{w\in W}\Zeta_w$. Find a countable subgroup $H\subset G$ such that $W\subset H*\langle x\rangle =H[x]$. By \cite{Zel}, the group $H$ is a quasitopological group with respect to some non-discrete Hausdorff (even regular) topology $\tau$.

We claim that for each monomial $w=a_0x^{\e_1}a_1x^{\e_2}a_2\in W\subset H^2[x]$ with $\e_1,\e_2\in\{-1,1\}$ the cozero set $\Zeta_w$ contains a neighborhood $U_w\in\tau$ of $1$. Since $1\in\Zeta_w$, we get $a_01^{\e_1}a_11^{\e_2}a_2\ne1$ and hence $a_01^{\e_1}a_1\ne a_2^{-1}1^{-\e_2}$. Since the topology $\tau$ on $H$ is Hausdorff, the points $a_01^{\e_1}a_1$ and $a_2^{-1}1^{-\e_2}$ have disjoint neighborhoods $O(a_01^{\e_1}a_1),O(a_2^{-1}1^{-\e_2})\in\tau$. Since $(H,\tau)$ is a quasitopological group, there is a neighborhood $U_w\in\tau$ of $1$ such that $a_0U_w^{\e_1}a_1\subset O(a_01^{\e_1}a_1)$ and $a_2^{-1}U_w^{-\e_2}\subset O(a_2^{-1}1^{-\e_2})$. For this neighborhood $U_w$ we get $1\notin a_0U_w^{\e_1}a_1U_w^{\e_2}a_2$ and $U_w\subset\Zeta_w$. Now we see from $\bigcap_{w\in W}U_w\subset \bigcap_{w\in W}\Zeta_w=\{1\}$ that $1$ is an isolated point of $H$ in the topology $\tau$, which is a contradiction.
\end{proof}

Theorem~\ref{t1} implies that for each infinite group $G$ the quasitopological group $(G,\Zeta_{G^2[x]})$ has infinite pseudocharacter. By the {\em pseudocharacter} $\psi_x(X)$ of a topological $T_1$-space at a point $x\in X$ we understand the smallest cardinality $|\U_x|$ of a family $\U_x$ of neighborhoods of the point $x$ such that $\cap\U_x=\{x\}$. The cardinal $\psi(X)=\sup_{x\in X}\psi_x(X)$ is called the {\em pseudocharacter} of the topological space $X$. 

Writing Theorem~\ref{t1} in terms of the pseudocharacter, we get

\begin{corollary}\label{c1} For any infinite group $G$ we get $$\psi(G,\Zeta_{G^1[x]})=|G|\mbox{  and }\psi(G,\Zeta_{G^2[x]})\ge\aleph_0.$$
\end{corollary}

It is interesting to observe that the lower bound $\psi(G,\Zeta_{G^2[x]})\ge\aleph_0$ in this corollary cannot be improved to $\psi(G,\Zeta_{G^2[x]})=|G|$.

\begin{theorem}\label{ex1} There is a group $G$ of cardinality continuum such that $\psi(G,\Zeta_{G^2[x]})=\aleph_0$. In fact, the group $G$ contains two disjoint countable subsets $A,B\subset G$ such that $\{1\}=\bigcap_{w\in W}\Zeta_w$ where $W=\{xax^{-1}b^{-1}:a\in A,\;b\in B\}$.
\end{theorem}

\begin{proof} Let $T=\bigcup_{n\in\w}\{0,1\}^n$ be the binary tree (consisting of finite binary sequences) and $\Aut(T)$ be its automorphism group (which has cardinality of continuum). The action of $\Aut(T)$ on $T$ induces an action of $\Aut(T)$ on the free group $F(T)$ over $T$. In the free group $F(T)$ consider the set $L$ of words of the form $al$ where $a\in T$ and $l=a\hat{\phantom{o}}0$ is the ``left'' successor of $a$ in the binary tree $T$. By symmetry, let $R\subset F(T)$ be the set of words $ar$ where $a\in T$ and  $r=a\hat{\phantom{o}}1$ is the ``right'' successor of $a$ in $T$. It is easy to check that the sets $L^{F(T)}=\{wxw^{-1}:x\in L,\; w\in F(T)\}$ and $R^{F(T)}=\{wxw^{-1}:x\in R,\; w\in F(T)\}$ are disjoint.

Let $F(T)\setminus\{1\}=\{w_n:n\in\w\}$ be an enumeration of the set of non-unit elements of $F(T)$. By induction for every $n\in\w$ we can find points $a_n,b_n\in F(T)$ such that
\begin{itemize}
\item $b_n=w_na_nw_n^{-1}$;
\item $a_n\notin\{b_i:i\le n\}$;
\item $b_n\notin\{a_i:i\le n\}$;
\item $a_n\notin R^{F(T)}$;
\item $b_n\notin L^{F(T)}$.
\end{itemize}

Consider the countable disjoint subsets $$A=L^{F(T)}\cup\{a_n:n\in\w\}\mbox{ and }B=R^{F(T)}\cup\{b_n:n\in\w\}$$ of the free group $F(T)$. 

Let $G$ be the semidirect product of the groups $F(T)$ and $\Aut(T)$. The elements of the group $G$ are pairs $(w,f)$, where $w\in F(T)$ and $f\in\Aut(T)$. The group operation on $G$ is given by the formula:
$$(w,f)\cdot(u,g)=(w\cdot f(u),fg).$$

The groups $F(T)$ and $\Aut(T)$ are identified with the subgroups $\{(w,1):w\in F(T)\}$ and $\{(1,f):f\in\Aut(T)\}$ of $G$.

We claim that for every non-unit element $x\in G\setminus\{1\}$ we get $xAx^{-1}\cap B\ne\emptyset$. If $x\in F(T)$, then $x=w_n$ for some $n\in\w$ and then $xAx^{-1}\cap B\ni b_n=w_na_nw_n^{-1}$. If $x\notin F(T)$, then $x=(w,f)$ for some non-identity automorphism $f\in\Aut(T)$. It follows that $f(t)\ne t$ for some node $t\in T$. We can assume that $t$ has the smallest possible height. The node $t$ is not the root of the tree $T$ because the automorphism $f$ of $T$ does not move the root. 
So, $t$ has an immediate predecessor $a$ in the tree $T$. The minimality of $t$ guarantees that $f(a)=a$ and hence $\{t,f(t)\}=\{l,r\}$ where $l=a\hat{\phantom{o}}0$ and $r=a\hat{\phantom{o}}1$ are the ``left'' and ``right'' successors of $a$ in the binary tree $T$. It follows that $al\in L\subset A$  and $f(al)=f(a)f(l)=ar\in R$. Consequently,
$$\begin{aligned}
x\cdot al\cdot x^{-1}&=(w,f)\cdot(al,1)\cdot(f^{-1}(w^{-1}),f^{-1})=\\
&=(w\cdot f(al),f\cdot 1)\cdot(f^{-1}(w^{-1}),f^{-1})=(w\cdot ar\cdot w^{-1},1)\in R^{F(T)}\subset B,
\end{aligned}
$$ witnessing that $xAx^{-1}\cap B\ne\emptyset$ and hence $1\in xAx^{-1}B^{-1}$.
\end{proof}

\begin{theorem}\label{t2} If $G$ is an infinite abelian group or a free group, then $\psi(G,\Zeta_{G[x]})=|G|$. Moreover, for any family $\U\subset\Zeta_{G[x]}$ of neighborhoods of 1 with $|\U|<|G|$ the intersection $\cap\U$ contains a subgroup $H\subset G$ of cardinality $|H|=|G|$.
\end{theorem}

\begin{proof} Let $\kappa=|G|$. In order to prove that $\psi(G,\Zeta_{G[x]})=|G|$, take any family of monomials $W\subset G[x]$ of size $|W|<\kappa$ with $1\in\bigcap_{w\in W}\Zeta_w$. It follows that $w(1)\ne 1$ for every $w\in W$ and hence the set $W(1)=\{w(1):w\in W\}$ is a subset of $G\setminus \{1\}$ and has cardinality $|W(1)|\le|W|<\kappa$.

If the group $G$ is abelian, then $G$ contains a subgroup $S$ of cardinality $|S|=\kappa$ that is isomorphic to the direct sum $\oplus_{\alpha\in \kappa}S_\alpha$ of cyclic groups, see \cite[\S16]{Fu}.  Since the set $S\cap W(1)$ has cardinality $<\kappa$, there is a subset $A\subset \kappa$ of cardinality $|A|<\kappa$ such that  $S\cap W(1)\subset\oplus_{\alpha\in A}S_\alpha$. Then the subgroup $H=\oplus_{\alpha\in \kappa\setminus A}S_\alpha$ has cardinality $|H|=\kappa$ and is disjoint with $W(1)$. Consequently, $H\cdot W(1)\subset G\setminus\{1\}$. We claim that $H\subset\bigcap_{w\in W}\Zeta_w$. Indeed, for every $w\in W$ and $x\in H$, by the commutativity of $G$ we get $w(x)=w(1)x^n$ for some $n\in\IZ$ and hence $w(x)\in w(1)\cdot H\subset W(1)\cdot H\subset G\setminus\{1\}$ witnessing that $H\subset \Zeta_w$.

Next, assume that $G=F(A)$ is a free group over an infinite alphabet $A\subset G$. It follows that the set $W(1)$ of cardinality $<\kappa$ lies in the subgroup $F(B)$ for some subset $B\subset A$ of cardinality $<\kappa$. Consider the subgroup $H=F(A\setminus B)$ and observe that $w(x)\ne 1$ for every $x\in H$. This means that $H\subset\bigcap_{w\in W}\Zeta_w$.

If $G=F(A)$ is a non-commutative free group over a finite alphabet $A$, then $G$ contains a free subgroup $F\subset G$ with infinitely many generators, see \cite[II.1.2]{OA}. By the preceding case the free group $F$ contains a subgroup $H\subset F$ of cardinality $|H|=\aleph_0=|G|$ such that $H\subset F\cap\bigcap_{w\in W}\Zeta_\w$. 
\end{proof}

It is interesting to compare Corollary~\ref{c1} and Theorem~\ref{ex1} with

\begin{theorem}\label{p1} For a finite subset $A$ of an infinite group $G$ and the family $$ 
W=G^1[x]\cup\{bxax^{-1}c,bx^{-1}axc:a\in A,\;b,c\in G\},$$ the group $G$ endowed with the topology $\Zeta_{W}$ is a quasitopological group with pseudocharacter $\psi(G,\Zeta_W)=|G|$.
\end{theorem}

\begin{proof} It follows from the definition of the family $W$ that the inversion $(\cdot)^{-1}:G\to G$ and left shifts $l_b:G\to G$, $b\in G$, are continuous with respect to the topology $\tau_W$. So, $(G,\tau_W)$ is a left-topological group with continuous inversion. The continuity of right shifts follows from the continuity of the left shifts and the continuity of the inversion.
Thus $(G,\tau_W)$ is a quasitopological group. 

In order to show that $\psi(G,\Zeta_W)=|G|$, fix a family $V\subset W$ of size $|V|<|G|$ with $1\in\bigcap_{v\in V}\Zeta_v$. Each monomial $v\in V$ is of the form $b_vxa_vx^{-1}c_v$ or $b_vx^{-1}a_vxc_v$ for some $a_v\in A$ and $b_v,c_v\in G$. Let $d_v=b_v^{-1}c_v^{-1}$ and observe that $b_vxa_vx^{-1}c_v\ne 1$ if and only if $xa_vx^{-1}\ne d_v$. Since $1\in\Zeta_v$, we get $a_v\ne d_v$ for every $v\in V$. 

Let $A=\{a_1,\dots,a_n\}$ be an enumeration of the finite set $A$. For every $k\le n$ let $D_k=\{d_v:v\in V,\; a_v=a_k\}$. It follows that  $a_k\notin D_k$ and $|D_k|\le|V|<|G|$.

Let $G_0=G$ and by induction for every $k<n$ define a subgroup $G_{k+1}$ of $G$ letting $G_{k+1}=G_k$ if the centralizer
$$Z_{G_k}(a_{k+1})=\{x\in G_k:xa_{k+1}=a_{k+1}x\}$$ has size $<|G|$ and $G_{k+1}=Z_{G_k}(a_{k+1})$ otherwise. 
So, $(G_k)_{k\le n}$ is a decreasing sequence of subgroups of $G$ having size $|G|$. 

Let $K$ be the set of positive numbers $k\le n$ such that $|Z_{G_{k-1}}(a_{k})|<|G|$.
For every $k\in K$ and $d\in D_k$ consider the set $X_{k,d}=\{x\in G_n:xa_kx^{-1}= d\}$. We claim that $|X_{k,d}|<|Z_{G_n}(a_k)|\le|Z_{G_{k-1}}(a_k)|<|G|$. 
Indeed, for any $x,y\in X_{k,d}$ we get $xa_kx^{-1}=d=ya_ky^{-1}$ and thus $y^{-1}xa_k=a_ky^{-1}x$, which implies $y^{-1}x\in Z_{G_n}(a_k)$ and thus $X_{k,d}\in yZ_{G_n}(a_k)$. Then the set
$$X_k=\bigcup_{d\in D_k}X_{k,d}\cup X_{k,d}^{-1}$$ has cardinality $|X_k|\le 2\cdot |D_k|\cdot |Z_{G_n}(a_k)|<|G|$. Finally, consider the set $Y=G_n\setminus\bigcup_{k\in K}X_k$ and observe that $Y=Y^{-1}$ and $|Y|=|G_n|=|G|$. 

We claim that $Y\subset\Zeta_v$ for every $v\in V$. In the opposite case $v(y)=1$ for some $v\in V$ and some $y\in Y$. It follows that $b_vya_vy^{-1}c_v=1$ or $b_vy^{-1}a_vyc_v=1$. Since $Y=Y^{-1}$, we loose no generality assuming that $b_vya_vy^{-1}c_v=1$ and hence $ya_vy^{-1}=d_v$. Find $k\le n$ such that $a_v=k$ and then $d_v\in D_k$. If $k\notin K$, then $y\in G_n\subset G_k=Z_{G_{k-1}}(a_k)$ and then $D_k\ni d_v=ya_vy^{-1}=ya_ky^{-1}=a_k$, which contradicts $a_k\notin D_k$. 
If $k\notin K$, then $ya_ky^{-1}=d_v\in D_k$ implies that $y\in X_k$ which contradicts the choice of $y\in Y=G_n\setminus \bigcup_{k\in K}X_k$.
\end{proof}

Proposition~\ref{p1} implies the following fact, proved in \cite{BL} by the technique of ultrafilters:

\begin{corollary} For any disjoint finite subsets $A,B$ of an infinite group $G$ the set $\{x\in G:xAx^{-1}\cap B=\emptyset\}$ is infinite.
\end{corollary}

\begin{remark} According to \cite{Pod} or \cite{Sipa2} each group $G$ with $\psi(G,\Zeta_{G[x]})=|G|$ is topologizable, that is, admits a non-discrete Hausdorff group topology.
\end{remark}

\begin{remark} In \cite{Hesse} groups $G$ with $\psi(G,\Zeta_{G[x]})\ge\kappa$ are called {\em $\kappa$-ungebunden}.
\end{remark}

Groups $G$ with non-discete $n$-th Zariski topology $\Zeta_{G^n[x]}$ have the following property:

\begin{theorem} If the $n$-th Zariski topology $\Zeta_{G^n[x]}$ on a group $G$ is non-discrete, then any finite subsets $A_0,A_1,\dots,A_n\subset G$ with $1\notin A_0A_1\cdots A_n$ there is an infinite symmetric subset $X=X^{-1}\ni 1$ such that $1\notin A_0XA_1X\cdots XA_n$.
\end{theorem} 

\begin{proof} By induction on $k\in\w$ we can construct a sequence $(x_k)_{k\in\w}$ of points of the group $G$, a sequence $(C_k)_{k\in\w}$ of finite subsets of $G$, and a sequence $(W_k)_{k\in\w}$ of finite subsets of $G^n[x]$ such that for every $k\in\IN$ the following conditions are satisfied:
\begin{enumerate}
\item $C_k=C_k^{-1}=(C_{k-1}\cup\{1,x_{k-1},x_{k-1}^{-1}\})^n\subset G$; 
\item $W_k=\{w\in C_k^n[x]:w(1)\ne 1\}$;
\item $w(x_k)\ne 1$ for any $w\in W_k$;
\item $x_k\notin\{x_i:i<k\}$.
\end{enumerate}
We start the inductive construction letting $x_0=1$ and $C_0=C_0^{-1}$ be any finite set containing $A_0\cup\dots\cup A_n$. For each $k$ the choice of the point $x_k$ is possible since the $n$th Zariski topology $\Zeta_{G^n[x]}$ is not discrete.

The conditions (1)--(3) of the inductive construction guarantee that the set $X=\{x_k,x_k^{-1}:k\in\w\}$ has the desired property: $1\notin A_0XA_1X\cdots XA_n$.
\end{proof}   

Looking at Theorem~\ref{t1} one may ask what happens with the Zariski topologies $\Zeta_{G^n[x]}$ for higher $n$. For sufficiently large $n$ the topology $\Zeta_{G^n[x]}$ can be discrete.

The following criterion is due to A.Markov \cite{Mar}:

\begin{theorem}[Markov] For a countable group $G$ the following conditions are equivalent:
\begin{enumerate}
\item the topology $\Zeta_{G[x]}$ is discrete;
\item the topology $\Zeta_{G^n[x]}$ is discrete for some $n\in\w$;
\item the group $G$ is non-topologizable.
\end{enumerate}
\end{theorem}

We recall that a group $G$ is {\em non-topologizable} if $G$ admits no non-discrete Hausdorff group topology. For examples of non-topologizable groups, see \cite{Shel}, \cite{Ol80}, \cite{Ol89}, \cite{MO}, \cite{Tro1}, \cite{Tro2}, \cite{KT}. In particular, the non-topologizable group from \cite{Ol80} yields the following example mentioned in \cite{KT}:

\begin{example} There is a countable group $G$ whose 665-th Zariski topology $\Zeta_{G^{665}[x]}$ is discrete.
\end{example}

What happens for $n<665$, in particular, for $n=3$?

\begin{problem} Is the 3-d Zariski topology $\Zeta_{G^3[x]}$ non-discrete on each infinite group $G$?
\end{problem}

Also Corollary~\ref{c1} and Example~\ref{ex1} suggest:

\begin{problem} Is $\psi(G,\Zeta_{G^2[x]})\ge\log|G|$ for every infinite group $G$?
\end{problem} 

\begin{remark} By \cite{Hesse} for every infinite cardinal $\kappa$ there is a  topologizable group $G$ with $|G|=\kappa$ and $\psi(G,\Zeta_{G[x]})=\aleph_0$.
\end{remark}

\end{document}